\documentclass[12pt,a4paper]{amsart}
\usepackage{graphicx,amssymb}
\usepackage[usenames]{color}
\input xy
\xyoption{all}
\usepackage[all]{xy}
\usepackage{hyperref}

\newtheorem{thm}{Theorem}[section]
\newtheorem{cor}[thm]{Corollary}
\newtheorem{lem}[thm]{Lemma}
\newtheorem{prop}[thm]{Proposition}
\theoremstyle{definition}

\newtheorem{rem}[thm]{Remark}

\numberwithin{equation}{section}

\newcommand{\ZZ}{\mathbb Z}
\newcommand{\CC}{\mathbb C}
\newcommand{\PP}{\mathbb P}

\newcommand{\lra}{\longrightarrow}

\newcommand{\ra}{\rightarrow}

\DeclareMathOperator{\Ind}{{Ind}}

 \DeclareMathOperator{\Ker}{Ker}

 \DeclareMathOperator{\Nm}{{Nm}}

 \DeclareMathOperator{\Ima}{{Im}}

\begin{document}

\title[ ]{Etale double covers of cyclic p-gonal covers}
\author{Angel Carocca, Herbert Lange and Rub\'{\i} E. Rodr\'iguez}
\address{Departamento de Matem\'atica y Estad\'istica, Universidad de La Frontera, Avenida Francisco Salazar 01145, Casilla 54-D, Temuco, Chile.}
\email{angel.carocca@ufrontera.cl}
\address{Department Mathematik, Universit\"at Erlangen, Cauerstrasse 11, 91058 Erlangen, Germany.}
\email{lange@math.fau.de}
\address{Departamento de Matem\'atica y Estad\'istica, Universidad de La Frontera, Avenida Francisco Salazar 01145, Casilla 54-D, Temuco, Chile.}
\email{rubi.rodriguez@ufrontera.cl}

\thanks{Partially supported by Anillo ACT 1415 PIA-CONICYT}

\begin{abstract}
This paper computes the Galois group of the Galois cover of the composition of  an \'etale
double cover of a cyclic $p$-gonal cover for any prime $p$. Moreover a relation between some of its Prym
varieties and the Jacobian of a subcover is given. In a sense this generalizes the trigonal construction.

\end{abstract}

\maketitle

\section{Introduction}
In this paper we investigate the Galois group of the Galois cover of the composition of \'etale double
coverings $Y \ra X$ of cyclic covers $X \ra \PP^1$ of prime degree $p$. For $p = 2$, Mumford shows in
\cite{m} that $Y \ra \PP^1$ is Galois with Galois group the Klein group of order 4 and
the Prym variety $P(Y/X)$ is isomorphic as a principally polarized abelian variety to either a Jacobian or the product of 2 Jacobians.

For $p=3$, the trigonal construction tells us the principally polarized $P(Y/X)$ is isomorphic to a
Jacobian of a tetragonal curve. In Section 2 we study the Galois group of the Galois closure
$Z \ra \PP^1$ of $Y \ra \PP^1$ for an odd prime $p$. The main result of this section is\\

{\bf Theorem 2.6}. {\it Let $p$ be an odd prime, $Y \ra X$ an \'etale double
cover and $X \ra \PP^1$ a cyclic cover of degree $p$. Then $Y \ra \PP^1$ is not Galois. Denoting by $Z \ra \PP^1$ its Galois closure, its Galois group $G$ is
$$G = N \rtimes P
$$
where 
$ \; N \cong \ZZ_2^{p-1} \; $ and $   \; P \cong \ZZ_p, \; $   and $X = Z/N$, $Y=Z/H$, with $H$ a maximal subgroup of $N$.}\\

There are $2^{p-1} -1$ maximal subgroups of $N$.  The group $P$ acts on them by conjugation
and there are $m := \frac{1}{p}(2^{p-1} -1)$ conjugacy clases of such subgroups.\\
Let $\{ Y_i \ra X \;|\; i = 1, \dots, m \}$ be the corresponding double covers. It is easy to see that they are all \'etale.  If $T:= Z/P$, there is a natural homomorphism
$$
\alpha: \prod_{i=1}^m P(Y_i/X) \ra JT.
$$
 Our main result is:\\

{\bf Theorem 3.1}. {\it $\alpha: \displaystyle{\prod_{i=1}^m P(Y_i/X)} \ra JT$ is an isogeny with kernel in the $2^{p-2}$-division points}. \\

As an immediate consequence we get examples of Jacobians with arbitrarily many isogeny factors of the same dimension.
For $p=3$ this is not yet the trigonal construction, which, however, is an easy consequence, as we show in Section 4.

For the sake of completeness, we also consider the case $p = 2$, i.e., we give a proof of Mumford's
theorem mentioned above. Note that Mumford gives only a short sketch of proof leaving the details to
the reader. It seems to us that our proof is different from the one Mumford had in mind.

\section{\'Etale covers of cyclic $p$-gonal covers}

\subsection{The structure of the Galois group}
Let $p$ be a prime and $\varphi: X \ra \PP^1$ be a cyclic covering of degree $p$ ramified over
$\beta$ points of $\PP^1$, with $\beta \geq 3$. Observe that if $p=2$ then $\beta$ must be even.

Let $\psi:Y \ra X$ be an \'etale double cover and $\widetilde \varphi: Z \ra \PP^1$ the Galois closure of the
composed map $\varphi \circ \psi$.
Let $G$ denote the Galois group of $\widetilde \varphi$
and $H$ and $N$ the subgroups of $G$ corresponding to $Y$ and $X$. So we have the following commutative diagram
\begin{equation} \label{dia:dia1}
\xymatrix{
& Z \ar[dl] \ar[ddd]^{\widetilde \varphi}\\
Y = Z/H \ar[d]_{2:1}^\psi &\\
X = Z/N \ar[dr]^\varphi_{p:1} &\\
& \PP^1
}
\end{equation}
In this section we determine the structure of $G$.

\begin{lem}  \label{lem1.1}
The
permutation representation $\rho$ of of $G$ on the right cosets of $H$ in $G$ has its image in the
alternating group $A_{2p}$ of degree $2p$, and the non-trivial elements of $G$ fixing points in $Z$ have order $p$. Moreover, the representation $\rho: G \ra A_{2p}$ is injective.
\end{lem}

\begin{proof}
Recall that $Y \ra X$ is the double covering corresponding to the embedding $H \subset N$.
Since $\varphi$ is cyclic of prime degree, the local monodromy of each of its branch points is a cycle of
length $p$. Since $\psi$ is an \'etale double cover,  every local monodromy of $\varphi \circ \psi$
above a branch point is the product  of two
disjoint cycles of length $p$ and hence in $A_{2p}$. Since $G$ is generated by these products, this gives the first assertion.
The second assertion is clear.
\end{proof}

\begin{cor} \label{c1.2}
If $p =2$, the covering $\varphi \circ \psi$ is Galois with Galois group $G$ the Klein group of order $4$.
In particular $Z = Y$.
\end{cor}

\begin{proof}
According to Lemma \ref{lem1.1}, $G$ is a subgroup of $A_4$, generated by elements which are products of  two disjoint cycles of length $ 2$. Hence $ G $ is the Klein group of order four.
\end{proof}

\begin{prop} \label{pr1.3}
Suppose $p$ is an odd prime. Then the covering $\varphi \circ \psi: Y \ra \PP^1$ cannot be Galois.
\end{prop}

\begin{proof}
Since  groups of order $2p$ cannot be generated by elements of order $p$,
the covering $\varphi \circ \psi: Y \ra \PP^1$ cannot be Galois.
\end{proof}

For the rest of this section we assume that $p$ is an odd prime; hence the covering $Y \ra \PP^1$ is not Galois, so $Z \neq Y$ and $H$ and $N$ are the proper subgroups of
$G$ corresponding to $Y$ and $X$ respectively, as in Diagram \eqref{dia:dia1}.

Let $\{1 = g_1, g_2, \dots , g_p\}$ denote a complete set of representatives of right cosets of $N$ in $G$ and  $\{1=n_1,n_2\}$ denote a
 complete set of representatives of right cosets of $H$ in $N$. Then the set $\{ n_ig_j: i = 1, \dots,p, j=1,2 \}$ is a  complete set of representatives of right cosets
of $H$ in $G$.

For $i = 1, \dots , p$ consider
$$
\Delta_i := \{H n_1g_i, H n_2g_i \}
$$
as a set of two elements. Then the right action of $G$ on the right cosets of $H$ in $G$ induces a
transitive action of $G$ on the set
$$
\Omega := \Delta_1 \cup \Delta_2 \cup \cdots \cup \Delta_p.
$$
This is just the right action of $G$ on the right cosets of $N$ in $G$. Now denote for $i = 1, \dots,p$,
\begin{equation} \label{eq1.1}
H_i:=g_i^{-1}Hg_i.
\end{equation}
Clearly each $H_i$ is a normal subgroup of index 2 in $N$.

\begin{lem} \label{l1.1}
\begin{enumerate}
\item[(i)] Any element of $H_i$ stabilizes each of the two points of $\Delta_i$;
\item[(ii)] $ N$ is the stabilizer of each set $\Delta_i$.
\end{enumerate}
\end{lem}

\begin{proof}
For (i) use that that $H$ is normal in $N$.
By definition, $N$ is the normal subgroup of $G$ corresponding to the covering $X \ra \PP^1$. Since
the $\Delta_i$ represent a right cosets of $N$ in $G$, this implies (ii). One can also see this  directly:
suppose $n \in N$. For any $i, 1 \le i \le p$ there
is an $n'_i \in N$ such that $n = g_i^{-1}n'_i g_i$. Then we have
$$
\Delta_i n = \{Hn_1g_i,Hn_2g_i\}g_i^{-1}n'_ig_i = \{Hn_1n'_ig_i, Hn_2n'_ig_i \} = \Delta_i.
$$
Since $ \; G \; $ does not stabilize $ \; \Delta_i \; $ and $ \; N \; $ has prime index in $ \; G, \; $ we conclude that $ \; N \; $ is the stabilizer of $ \; \Delta_i.$
\end{proof}

Recall the representation $\rho: G \ra A_{2p}$. Since $N$ is a normal subgroup of index $p$ in $G$,
we may enumerate the right cosets $\Delta_i$ of $N$ in $G$ in such a way that we can identify the
set $\Delta_i$ with the set $\{i,p+i\}$ and the action of $G$ on the $\Delta_i$ corresponds to
the permutation (right-)action of the group $A_{2p}$ on the set $\{1, \dots, 2p\}$.
Moreover, fixing a branch point, we may enumerate its branches  in such a way that the
local monodromy around this point
is given by the cycle
$$
\sigma := (1,2,\dots,p)(p+1,p+2,\dots,2p).
$$

\begin{lem} \label{l1.2}
$$
N \cong ({\mathbb Z}_ 2)^{p-1}.
$$
\end{lem}

\begin{proof}
Consider, for $i = 1, \dots, p, \; $ the transposition
$t_i:= (i, p+i)$. Certainly $t_i$ is not contained in $G$, since $G \subset A_{2p}$. However, we have
$$
s_1:= t_1t_2 \in N,
$$
since it stabilizes each set $\{(i, p+i)\}$ and so is in $N$ by Lemma \ref{l1.1} and the identifications.
Moreover,
\begin{eqnarray}
\nonumber \sigma^{-1} t_1t_2 \sigma &=& t_2t_3 =: s_2 \in N\\
\nonumber \sigma^{-2} t_1t_2 \sigma^2 &=& t_3t_4 =: s_3 \in N\\
\cdots &&  \label{e1.1}\\
\nonumber \cdots &&\\
\nonumber \sigma^{-(p-1)} t_1t_2 \sigma^{p-1} &=& t_pt_1 =:s_p \in N
\end{eqnarray}
which gives
$$
\prod_{i=1}^p s_i = (t_1t_2)(\sigma^{-1} t_1t_2 \sigma)(\sigma^{-2} t_1t_2 \sigma^2) \cdots
(\sigma^{-(p-1)} t_1t_2 \sigma^{p-1}) = 1.
$$
Since the cycles $s_i$ pairwise commute, and clearly there is no non-trivial relation between the cycles
$s_1, s_2, \dots s_{p-1}$, this implies
$$
2^{p-1} = |\langle s_1,s_2, \dots, s_p \rangle | \leq  |N|.
$$
Since a non-trivial  element of $ \; H_1 \cap H_2 \cap ...\cap H_{p-1} \; $ would be the transposition exchanging the two points of $ \; \Delta_p \; $, which is not in $ \; A_{2p} \; $, we have
  $ \; H_1 \cap H_2 \cap ...\cap H_{p-1}  = \{ 1 \}.  \; $
 Consider the group homomorphism $$ \; \Phi' : N \rightarrow N/H_1 \times N/H_2 \times \ldots \times N/H_{p-1} \; $$ defined by
$ \; \Phi'(n) = (H_1n, H_2n, \ldots , H_{p-1}n). \; $
Since $ \; \ker(\Phi') = \displaystyle\bigcap_{i=1}^{p-1}H_i = \{ 1 \}, \; $ we have  $ \; N \cong N/ \ker{\Phi'} \lesssim ({\mathbb Z}_2)^{p-1}.\; $ Hence
$ \; N \cong  ({\mathbb Z}_2)^{p-1}.\; $
\end{proof}

\begin{thm} \label{thm1.6}
Let $X \ra \PP^1$ be a cyclic covering of degree an odd prime $p$, and $Y \ra X$ be an \'etale double covering.
Let $Z \ra \PP^1$ be the  Galois closure of the composition $Y \ra \PP^1$ with Galois group $G$.
With the notation of above, if $P$ denotes the subgroup of $G$ generated by the cycle $\sigma$, then $G$ is the semi-direct product
$$
G = N \rtimes P \simeq \ZZ_2^{p-1}  \rtimes \ZZ_p.
$$
\end{thm}

\begin{proof}
Since $ \; N \; $ is a normal subgroup of index $ \; p \; $ in $ \; G \; $ and $ \; \vert N \vert = 2^{p-1} \; $ we have $ \; G = N \rtimes P.$
\end{proof}

A presentation of $G$ is given as
\begin{equation} \label{eq:Gpres}
  G = \langle s_1, \dots, s_{p}, \sigma \;|\; \prod_{i=1}^p s_i =1,  \sigma^p = 1, \; s_1^2 =1, \; \sigma^{-1}s_j \sigma = s_{j+1} \; \mbox{for} \;  j = 1, \dots,p -1\rangle.
\end{equation}

\subsection{The subcovers of $Z$}
Let $p$ denote an odd prime. According to a well-known result of elementary number theory, the number
\begin{equation}  \label{e1.3}
m:= \frac{1}{p}(2^{p-1} -1)
\end{equation}
is an integer. The abelian group $ \; N \; $ has  exactly  $ \; m \; $ conjugacy classes of maximal subgroups with respect to the action of $ \; P. \; $  For each $ \; 1 \leq j \leq m \; $ consider  $\; R_j \; $  a representative of the corresponding conjugacy class of maximal subgroups of  $ \; N.\; $ Here $ \; R_1 = H. \; $
\vspace{1mm}\\
To each  subgroup $ \; R_j \; $ corresponds a double covering of $X$.
Let
$$
Y_j := Z/R_j\quad  \mbox{for} \quad j=1,\dots,m
$$
denote the corresponding curves. In particular, $Y_1 = Y$.

Denoting moreover $T := Z/P$, we have the following diagram
\begin{equation} \label{d1.4}
\xymatrix{
&&& Z \ar[ddl] \ar[ddlll]_{2^{p-2}:1} \ar[ddll] \ar[dr]^{p:1}\\
&&&& T \ar[dddll]^{2^{p-1}:1}\\
Y_1  \ar[dr]_{2:1} & \cdots  \ar[d] & Y_m \ar[dl]\\
&X  \ar[dr]_{p:1} &\\
&& \PP^1
}
\end{equation}

\begin{lem} \label{l1.4}
The map $Z \ra X$ of degree $2^{p-1}$ is \'etale. In particular, all covers $Y_i \ra X$ of the above diagram are \'etale.
\end{lem}

\begin{proof}
This follows immediately from Lemma \ref{lem1.1}
\end{proof}

\begin{prop}
Let $\beta$ denote the number of branch values of $X \ra \PP^1$, with $\beta \geq 3$. Then the
genera of the curves in the above diagram are:
\begin{itemize}
\item $g(X) = \frac{p-1}{2}(\beta - 2)$;
\item $g(Y_i) = (p-1)(\beta - 2) -1$;
\item $g(Z) = 2^{p-2}(p-1)\beta - (p2^{p-1} -1)$;
\item $g(T) = \frac{2^{p-1} -1}{p}(\frac{p-1}{2}\beta -p)$.
\end{itemize}
\end{prop}

\begin{proof}
The first 3 assertions are obvious, since $X \ra \PP^1$ is totally ramified and $Z\ra X$ is \'etale.
For the last assertion note that over each branch value of $T \ra \PP^1$ there are
$m = \frac{2^{p-1} -1}{p}$ branch points of index $p-1$ and one point \'etale over $\PP^1$.
So the Hurwitz formula gives the assertion.
\end{proof}

As an immediate consequence we obtain the following result.

\begin{cor} \label{cor1.6}
If $P(Y_i/X)$ denotes the Prym variety of $Y_i/X$, we have
$$
\sum_{i=1}^m \dim P(Y_i/X) = \dim JT.
$$
\end{cor}

This suggests that there is a relation between the Prym varieties $P(Y_i/X)$ and the Jacobian $JT$.
The aim of this paper is to study the relation.

\subsection{The rational representations of $G$}
We follow \cite[Section 8.2]{s} to determine the irreducible representations of a
semidirect product $G = N \rtimes P. \; $ Let $\widehat N$ be the character group of $N$. The group $P$ acts on
$\widehat N$ in the usual
way. The stabilizer in  $ \; P \; $ of the trivial character $\chi_0$  of $ \; N \; $ is $P$ itself, whereas the stabilizer of any non-trivial complex irreducible 
character of $ \; N \; $ consists of  $ \; \{ 1\} $ only. Hence there are exactly $1+m$ orbits of the action of $P$ on $\widehat N$, with $m$ as in \eqref{e1.3}. Let $\chi_0, \chi_1, \dots, \chi_m$ be a system of representatives of these orbits,  $ \; \rho_0, \dots, \rho_{p-1}$ ($\rho_0$ the trivial representation)  the irreducible representations of the cyclic group $P \; $ and $ \; \eta \; $ the trivial character of $ \; \{ 1\} .$ 
\\
According to \cite[Proposition 25]{s}
$$ \{ \chi_0 \otimes \rho_j , \; \; \Ind_N^G(\chi_i \otimes \eta) \; \; / \: \; 0 \leq j \leq p-1, \; \; 1 \leq i \leq m \}$$
is the set of all complex irreducible representations of $G.$ The next result follows immediately 
\begin{cor}
The rational irreducible representations of $G$ are exactly the trivial representation $\rho_0 =  \chi_0 \otimes \rho_0$,
the representations $\theta_i =  \Ind_N^G(\chi_i \otimes \eta) \; $ of degree
$\; p$ for $i = 1, \dots, m$, and the representation $\psi := (\chi_0 \otimes \rho_1)  \oplus \dots  \oplus (\chi_0 \otimes \rho_{p-1}) $ of degree $p-1$.
\end{cor}

According to \cite[Proposition 13.6.1]{bl} the rational irreducible representations of $G$ correspond
canonically and bijectively to a set of $G$-stable abelian subvarieties of the Jacobian $JZ$ of $Z$
such that the
addition map is an isogeny. If the abelian subvariety of $JZ$ corresponding to the rational irreducible
representation $\rho$ of $G$ is denoted by $J_\rho$, the isotypical decomposition of $JZ$ is the
isogeny given by the addition map
$$
J_{\rho_0} \times J_\psi \times J_{\theta_1} \times \cdots \times J_{\theta_m} \ra JZ.
$$

Furthermore, according to \cite[Proposition 13.6.2]{bl} and \cite{cr2}, for each rational irreducible representation $\rho$ of $G$ there exist abelian subvarities $B_{\rho}$ of $J_\rho$ such that $B_{\rho}^{n_{\rho}}$ is isogenous to $J_\rho$, with
$$
n_{\rho} = \frac{\dim V_{\rho}}{m_{\rho}} ,
$$
where $V_{\rho}$ is a complex irreducible representation of $G$ Galois associated to $\rho$ and $m_{\rho}$ is the Schur index of $V_{\rho}$.

The subvarieties  $B_{\rho}$ are, in general, determined only up to isogeny, with  $B_{\rho_0}=J_{\rho_0} = J(Z/G)$. In our case, we have $ \; m_{\rho_0} = m_{\psi} = m_{\theta_j} = 1, \; \; \dim(V_0) = \dim(V_{\psi}) = 1 \; $ and $ \; \dim(V_{\theta_j}) = p, \; $ hence

$$
B_\psi  = J_\psi \ \  ,  \  \   J_{\theta_j}  \sim B_{\theta_j}^p  \ \  , \ \  J_{\rho_0} = J\PP^1 =0
$$
where $\sim$ denotes isogeny.

Furthermore, it follows from \cite[Corollary 5.6]{cr2} that
$$
 B_{\theta_j} \sim P(Y_j/X) \ \ \textup{and } \ \   J_\psi \sim J(X) .
$$

 Therefore the group algebra decomposition of $JZ$ is given by

$$
JX \times \prod_{j=1}^m P(Y_j/X)^p  \ra JZ .
$$

\section{The isogeny $\alpha$}

Let the notation be as in Section 1 and for $i = 1, \dots, m$ denote
$$
\nu_i: Z \ra Y_i \quad \mbox{and} \quad \mu: Z \ra T,
$$
the maps of diagram
\eqref{d1.4}, so that $\nu_i^*:JY_i \ra JZ$ and $\Nm \mu: JZ \ra JT$ are the induced homomorphisms
of the corresponding Jacobians. Then the addition map gives a homomorphism
\begin{equation} \label{e2.1}
\alpha:= \sum_{i=1}^m \Nm \mu \circ \nu_i^* : \; \prod_{i=1}^m  P(Y_i/X) \ra JT.
\end{equation}

According to Corollary \ref{cor1.6}, $\displaystyle\prod_{i=1}^m P(Y_i/X)$ and $JT$ are of the same dimension.
The aim of this section is the proof of the following theorem.

\begin{thm} \label{thm2.1}
$\alpha : \displaystyle\prod_{i=1}^m  P(Y_i/X) \ra JT$ is an isogeny with kernel contained in the $2^{p-2}$-division points.
\end{thm}

For this we use the following result (for the proof see \cite[Corollary 2.7]{rr}).

\begin{prop} \label{p2.3}
Let $f: Z \ra X := Z/N$ be a Galois cover of smooth projective curves with Galois group
$N$  and $H \subset G$ a subgroup. Denote by $\nu: Z \ra Y := Z/H$ and
$\varphi: Y \ra X$  the corresponding covers. If $\{g_1, \dots, g_r\}$ is a complete set of
representatives of $G/H$, then we have
$$
\nu^*(P(Y/X)) = \{ z \in JZ^H \;|\; \sum_{i=1}^r g_i(z) = 0 \}^0.
$$
\end{prop}

Now denote for $i = 1, \dots, m$,
$$
A_i := \nu_i^*(P(Y_i/X))
$$
and let
$$
A := \sum_{i=1}^m A_i \quad \mbox{and} \quad B:= \mu^*(JT).
$$
Recall from \eqref{eq:Gpres} that $G = N \rtimes P$ with
$$
N = \left\{ \prod_{i=1}^{p-1} s_i^{j_i} \;|\; 0 \leq j_i \leq 1, i =1, \dots ,p-1 \right\}
\quad \mbox{and} \quad P = \langle \sigma \rangle
$$
with $s_i$ and $\sigma$ as in Section 1.
The group $P$ acts by conjugation on the elements of $N$ by
\begin{equation} \label{e2.4}
\sigma^{-1} s_i \sigma = s_{i+1} \quad \mbox{for} \quad i = 1, \dots,p-1 \quad \mbox{with} \quad
s_p = \prod_{i=1}^{p-1} s_i.
\end{equation}
Recall furthermore that $R_i$ is the subgroup of $N$ giving the cover $Y_i \ra X$.
Then it is easy to see that we have the following commutative diagram
\begin{equation} \label{d2.5}
\xymatrix{
&A \ar[rr]^{\sum_{i=0}^{p-1}  \sigma^i}  \ar[dr]_{\Nm \mu} && B \ar[rr]^{\sum_{i=1}^m \sum_{h \in R_i} h} \ar[dr]_{\beta} && A\\
\prod_{i=1}^m P(Y_i/X)   \ar[rr]_{\alpha} \ar[ur]^{\sum_{i=1}^m \nu_i^*} && JT \ar[rr]_{\beta \circ \mu^*} \ar[ur]_{\mu^*} && \prod_{i=1}^m P(Y_i/X) \ar[ur]_{\sum_{i=1}^m \nu_i^*}
}
\end{equation}
with $\beta = (\Nm \nu_1, \Nm \nu_2, \dots, \Nm \nu_m)$.

For $i = 1, \dots,m$ consider the following subdiagram
\begin{equation} \label{d2.4}
\xymatrix{
&A_i \ar[rr]^{\sum_{i=0}^{p-1}  \sigma^i}  \ar[dr]_{\Nm \mu} && B_i \ar[rr]^{\sum_{h \in R_i} h} \ar[dr]_{\Nm \nu_i} && A_i\\
P(Y_i/X)   \ar[rr]_{\alpha_i} \ar[ur]^{ \nu_i^*} && C_i \ar[rr]_{\Nm \nu_i \circ \mu^*} \ar[ur]_{\mu^*} && P(Y_i/X) \ar[ur]_{\nu_i^*}
}
\end{equation}
with $\alpha_i:= \Nm \mu \circ \nu_i^*, \; C_i := \Nm \mu(A_i)$ and $B_i := \mu^*(C_i)$.

\begin{prop}  \label{prop3.3}
For $i = 1, \dots, m$ the map $\Nm \nu_i \circ \mu^* \circ \alpha_i : P(Y_i/X) \ra P(Y_i/X)$
is multiplication by $2^{p-2}$.
\end{prop}

\begin{proof}
Since $\nu_i^*: P(Y_i/X) \ra A_i$ is an isogeny, it suffices to show that the composition
$$
\Phi_i: = \sum_{h \in R_i} h \circ \sum_{i = 0} ^{p-1} \sigma^i: A_i \ra A_i
$$
is multiplication by $2^{p-2}$.

Now from Proposition \ref{p2.3} we deduce
\begin{equation} \label{e2.7}
A_i = \{ z \in JZ \;|\; hz = z \; \mbox{for all} \; h \in R_i \; \mbox{and} \; nz = -z \; \mbox{for all} \; n \in N \setminus R_i  \}^0
\end{equation}
since any $n \in N \setminus R_i$ induces the non-trivial involution of $Y_i/X$ and $A_i$ is the image
of $P(Y_i/X)$.

Now for any $z \in A_i$,
$$
\Phi_i(z) = \sum_{h \in R_i} h(z) + \sum_{h \in R_i} h \sum_{k=1}^{p-1} \sigma^k(z).
$$
By equation \eqref{e2.7}  we have
$$
\sum_{h \in R_i} h (z) = |R_i| z = 2^{p-2}z
$$
and for $k = 1,\dots,p-1$,
$$
\sum_{h \in R_i} h \sigma^k(z) = \sigma^k \sum_{h \in \sigma^{-k} R_i \sigma^k} h(z) = 0,
$$
since $R_i \neq \sigma^{-k} R_i \sigma^k$ and considering that  half of the elements of the
subgroup $\sigma^k R_i \sigma^k$ belong to $R_i$, hence fix $z$, and the other half belongs
to $N \setminus R_i$ and hence sends $z$ to $-z$. Together this completes the proof of the proposition.
\end{proof}

\begin{proof}[Proof of Theorem \ref{thm2.1}]
Since
$$
\beta \circ \mu^* \circ \alpha = \prod_{i=1}^m (\Nm_{\nu_i} \circ \mu^* \circ \alpha_i),
$$
Proposition \ref{p2.3} implies that $\beta \circ \mu^* \circ \alpha$ is multiplication by $2^{p-2}$.
In particular $\alpha$ has finite kernel. But according to Corollary \ref{cor1.6},
$\displaystyle\prod_{i=1}^m P(Y_i/X)$ and $JT$ have the same dimension. So $\alpha$ is an isogeny.
\end{proof}

\begin{cor} \label{c3.4}
Given any positive integer $N$, there exist smooth projective curves $Y$ whose Jacobian is isogenous to the product of $m \ge N$ principally polarized abelian varieties  of the same dimension.
\end{cor}

\begin{proof}
Choose a prime $p$ such that $\frac{1}{p}(2^{p-1} -1) \ge N$. This is equivalent to $2^{p-1} > pN$. Hence there are infinitely many primes with this  property. According to Theorem \ref{thm2.1}, the Jacobian $JT$ has the property of the corollary.
 \end{proof}

We thank the referee for suggesting the following remark and Elham Izadi for the contents of it.

\begin{rem}
There is a slight relation of Corollary \ref{c3.4} and a question of Ekedahl and Serre, whether for any positive integer $g$ there 
is a smooth curve of genus $g$ whose Jacobian is isogenous to a product of elliptic curves. Izadi showed in \cite{i} that, if there 
is a complete subvariety of codimension $g$ in the moduli space, then there exist smooth curves of genus $g$ whose Jacobian is 
isogenous to the product of elliptic curves. As was later proved by Keel and Sadun however in \cite{ks}, there are no such subvarieties in characteristic 0.

\end{rem}

\section{The case $p =3$}

In this case we have $m = 1$, so let  $Y_1 =: Y$,  $\nu_1 =: \nu$ and $A_1 =:A$.
Moreover, the subgroup $N$ is the Klein group of order $4$. Diagram \eqref{d1.4} simplifies to
\begin{equation} \label{d3.1}
\xymatrix{
& Z  \ar[ddl]_\nu^{2:1} \ar[dr]^\mu_{3:1}&\\
&& T \ar[dddl]^{4:1}\\
Y  \ar[d]_{2:1}  &&\\
X  \ar[dr]_{3:1} &&\\
& \PP^1 &
}
\end{equation}

\begin{thm} \label{thm3.1}
The map $\alpha = \nu^* \circ \Nm \mu: P(Y/X) \ra JT$ is an  isogeny with kernel the group $P(Y/X)[2]$ of all two-division points.
\end{thm}

\begin{proof}
From Theorem
\ref{thm2.1} we know that $\ker \alpha \subseteq  P(Y/X)[2]$. On the other hand, $\mu^*$ is injective, since
$\mu: Z \ra T$ is ramified. Hence from diagram \eqref{d2.4} we have $\ker (\Nm \mu|_A) = \ker (1 + \sigma + \sigma^2)|_A)$.
So we get
$$
\ker \alpha = \{ z \in P(Y/X)[2] \;|\; (1 + \sigma + \sigma^2)(\nu^*(z)) = 0 \}.
$$
Let $\gamma: Y \ra X$ denote the double covering and $\epsilon: Z \ra X$ the composition
$$
\epsilon = \gamma \circ \nu.
$$
Since $N$ is a normal subgroup of $G$, the automorphism $\sigma$
descends to an automorphism $\overline \sigma: X \ra X$, also of order 3. This is the automorphism
giving the cyclic covering $X \ra \PP^1$.

Suppose $\eta$ is the two-division point of $JX$ giving the double cover $\gamma$ and let
$\eta^\bot$ be the subgroup of $JX[2]$ orthogonal with respect to the Weil form $e_{2\lambda}$
associated to twice the canonical polarization $\lambda$ of $JX$. Then from \cite{m} we know that
$$
P(Y/X)[2] = \gamma^*(\eta^\bot).
$$
This gives
\begin{eqnarray*}
\ker \alpha &=& \gamma^*\{x \in \eta^\bot \;|\; (1+\sigma + \sigma^2)\epsilon^*(x) =0 \}\\
&=& \gamma^*\{x \in \eta^\bot \;|\; \epsilon^* (1+\overline \sigma + \overline \sigma^2)(x) =0\}.
\end{eqnarray*}
But $JX = \ker (1 + \overline \sigma + \overline \sigma^2)$. In particular for all $x \in \eta^\bot$
we have $\epsilon^* (1+\overline \sigma + \overline \sigma^2)(x) =0$. Together this implies
$\ker \alpha = P(Y/X)[2]$.
\end{proof}

As an immediate consequence we get a version of the trigonal construction in the special case of an \'etale cover of a cyclic trigonal cover $X \ra \PP^1$.

\begin{cor} \label{c4.2}
Let the notation be as in Theorem \ref{thm3.1}. The isogeny $\alpha: P(Y/X) \ra JT$ induces an isomorphism of principally polarized abelian varieties
$$
\overline \alpha: \widehat{P(Y/X)} \ra JT
$$
where  $\;\widehat{}$ denotes the dual abelian variety.
\end{cor}

\begin{proof}
Let $\lambda_P$ denote the polarization on $P(Y/X)$ induced by the canonical polarization of $JY$.
It is twice a principal polarization. According to Theorem \ref{thm3.1}, $\alpha$ has kernel
$P(Y/X)[2]$ which coincides with the kernel of the polarization $\lambda_P$. Hence $\alpha$ factors as follows, with  $\overline \alpha \; $ an isomorphism,
\begin{equation} \label{d2.3}
\xymatrix{
P(Y/X) \ar[r]^\alpha  \ar[d]_{\lambda_P}  & JT\\
\widehat{P(Y/X)} \ar[ur]_{\overline \alpha}^\simeq \ar[ur]
}
\end{equation}
It remains to show that $\overline \alpha$ respects the principal polarizations. If we denote by $\lambda_1$ the polarization of $\widehat{P(Y/X)}$ induced via $\overline \alpha$ from
the canonical polarization $\lambda_{JT}$ of $JT$, we may complete diagram \ref{d2.3} to the following one.
\begin{equation*}
\xymatrix{
P(Y/X) \ar[r]^\alpha  \ar[d]_{\lambda_P}  & JT \ar[d]^{\lambda_{JT}}\\
\widehat{P(Y/X)} \ar[d]_{\lambda_1} \ar[ur]_{\overline \alpha}^\simeq \ar[ur] & \widehat{JT} \ar[dl]_\simeq^{\widehat{\overline \alpha}}\\
P(Y/X)
}
\end{equation*}
It now follows from the commutativity of this diagram that $\lambda_1$ is principal
and that $\ker(\lambda_1 \circ \lambda_P) = P(Y/X)[2]$. Hence $\lambda_1$ is the canonical principal polarization on $\widehat {P(Y/X)}$ as claimed.
\end{proof}

\section{Estimate of the kernel of $\alpha$ for odd $p$}

We show that the same proof as in the last section gives for any odd prime $p$ a lower bound for the order of the kernel of the isogeny $\alpha: \displaystyle\prod_{i=1}^m P(Y_i/X) \ra JT$. We have the following result.

\begin{prop}
With the notation of above we have for any odd prime $p$,
$$
\prod_{i=1}^m P(Y_i/X)[2]\; \subset \;\Ker \alpha \; \subset \; \prod_{i=1}^m P(Y_i/X)[2^{p-2}].
$$
Furthermore, for $p >3$, $\Ker \alpha$ cannot be equal to $\displaystyle\prod_{i=1}^m P(Y_i/X)[2]$.
\end{prop}

\begin{proof}
For the first assertion it suffices to see that $\Ker \alpha_i$ contains $P(Y_i/X)[2]$. But since $\mu^*$ is injective, $\mu$ being ramified of prime degree, it follows from diagram \eqref{d2.4} and Theorem \ref{thm2.1} that
$$
\Ker \alpha_i = \{z \in P(Y_i/X)[2^{p-2}] \;|\; \sum_{i=0}^{p-1} \sigma^i (\nu_i^* (z)) =0 \}
$$
Hence for the proof of the first assertion is suffices to show that for any $z \in P(Y_i/X)[2]$ we have
$$
\sum_{i=0}^{p-1} \sigma^i(\nu_i^*(z)) =0.
$$
This follows with the same proof as in the proof of Theorem \ref{thm3.1} for $p=3$.

Finally, if we had $\Ker \alpha = \displaystyle\prod_{i=1}^m P(Y_i/X)[2]$, the same proof as for Corollary \ref{c4.2} would provide an isomorphism of principally polarized abelian varieties $\displaystyle\prod_{i=1}^m \widehat{P(Y_i/X)} \simeq JT$.
For $p > 3$, i.e. $m >1, \; $ this contradicts the fact that the canonical polarization of $JT$ is irreducible.
\end{proof}

\section{The case $p = 2$}

Let $Y \ra X$ be an \'etale double covering of a double covering $X \ra \PP^1$. According to
Corollary  \ref{c1.2}, the composition $Y \ra \PP^1$ is Galois, with  Galois group  the Klein group
$$
G = \langle r, s \;|\; r^2 = s^2 = (rs)^2 = 1 \rangle.
$$
Denoting $Y_r := Y/\langle r \rangle$ and similarly $Y_s$ and $Y_{rs}$, we have the following diagram of double coverings,
\begin{equation}\label{d6.1}
\xymatrix{
& Y \ar[dl]_{\nu_s} \ar[dr]^{\nu_{rs}} \ar[d]^{\nu_r} &\\
X =Y_s \ar[dr]_{2:1} & Y_r \ar[d] & Y_{rs} \ar[dl]\\
& \PP^1 &\\
}
\end{equation}
We assume that $\nu_s$ is \'etale and $Y_s \ra \PP^1$ is ramified over $2\beta$ points of $\PP_1$, with $\beta \geq 3$ (so that $\dim P(Y/Y_s) >0$).
Each branch point of $Y_s \ra \PP^1$ is a branch point of exactly one of the maps $Y_r \ra \PP^1$ and
$Y_{rs} \ra \PP^1$. So if $2\beta_r$ respectively $2\beta_{rs}$ denote the number of branch points
of $Y_r \ra \PP^1$ respectively $Y_{rs} \ra \PP^1$, we have
$$
\beta = \beta_r + \beta_{rs}.
$$
The genera of the curves are:
$$
g(Y_s) = \beta -1; \quad g(Y) = 2\beta-3; \quad  g(Y_r) = \beta_r -1; \quad g(Y_{rs}) = \beta_{rs} -1.
$$
In particular, $\dim P(Y/Y_s) = g(Y_r) + g(Y_{rs})$.

\begin{prop} \label{p4.1}
The following map is an isogeny,
$$
\alpha: JY_r \times JY_{rs}  \ra P(Y/Y_s), \quad (x_1,x_2) \mapsto \nu_r^*(x_1) + \nu_{rs}^*(x_2)
$$
with kernel consisting at most of two-division points.
\end{prop}

\begin{proof}
First we claim that $\Ima(\alpha) \subset P(Y/Y_s)$. Note first that the automorphism $s$ descends  to an automorphism $\overline s$ of $Y_r$ and we have for any $x \in JY_r$
$$
s(\nu_r^*(x)) = \nu_r^*(\overline s(x)) = - \nu_r^*(x)
$$
where the last equation follows from Proposition \ref{p2.3}.
An analogous equation is valid for $\nu_{rs}^*$. So we have
$$
(1 + s)(\alpha(x_1,x_2)) = (1+s)(\nu_r^*(x_1) + \nu_{rs}^*(x_2)) = x_1 -x_1 + x_2 -x_2 =0,
$$
which implies the assertion.

It remains to show that $\ker \alpha$
consists of $2$-division points, since
$g(Y_r)+ g(Y_{rs}) = \dim P(Y/Y_s)$. For this it suffices to show that the composed map
$$
JY_r \times JY_{rs} \stackrel{\nu_r^* + \nu_{rs}^*}{\lra} JY \stackrel{(\Nm \nu_r,\Nm \nu_{rs})}{\lra} JY_r \times JY_{rs}
$$
is multiplication by 2. But $\Nm \nu_r \circ \nu_r^* = \deg \nu_r = 2$ and the same is valid for $\nu_{rs}$. This completes the proof of the proposition.
\end{proof}

Proposition \ref{p4.1} implies
\begin{eqnarray*}
\ker \alpha & = &  \{ (x_1,x_2) \in JY_r[2] \times JY_{rs}[2] \; |\; \nu_r^*(x_1) =
\nu_{rs}^*(x_2) \} \\
& = & (\nu_r^* \times \nu_{rs}^*)^{-1} \{(x,x) \in JY \times JY \;|\; x \in \nu_r^*JY_r[2] \cap
\nu_{rs}^*JY_{rs}[2] \}
\end{eqnarray*}
Since $\nu_r$ and $\nu_{rs}$ are ramified, the homomorphisms  $\nu_r^*$ and $\nu_{rs}^*$ are
injective.
Hence we get
\begin{equation} \label{e4.1}
\deg \alpha = | \nu_r^*JY_r[2] \cap \nu_{rs}^*JY_{rs}[2] |.
\end{equation}
The following theorem is due to Mumford (see \cite[page 356]{m}).
\label{thm4.2}
\begin{thm}
Let the notation be as above. Then
we have:
\begin{enumerate}
\item[(i)] the map
$$
\alpha:  JY_r \times JY_{rs} \ra P(Y/Y_s)
$$
is an isomorphism;
\item[(ii)] the isomorphism $\alpha$ respects the canonical principal polarizations.
\end{enumerate}
\end{thm}

\begin{proof} (i):
According to \eqref{e4.1} it suffices to show that the images of $JY_r[2]$ via $\nu_r^*$ and
$JY_{rs}[2]$ via $\nu_{rs}^*$ in $JY[2]$ intersect only in $0 \in JY$. Now,  fixing a theta
characteristic of JY, the 2-division points of $JY$ correspond in a natural way bijectively to the
theta characteristics of $Y$. An analogous statement is valid for $JY_r$ and $JY_{rs}$. Using this,
the assertion follows from the fact that the theta characteristics of $Y$ which are pullbacks from
theta characteristics of $Y_r$ are disjoint from those which are are pullbacks from theta characteristics
of $Y_{rs}$.

But this follows from the fact that, according to what we have said right after the diagram \ref{d6.1}, the
branch points $b_1, \dots, b_{2\beta}$ of $Y_s \ra \PP^1$ can be enumerated in such a way
that $b_1, \dots, b_{2\beta_r}$ are the branch points of $Y_r \ra \PP^1$ and that
$b_{2\beta_r +1}, \dots b_{2\beta}$ are the branch points of $Y_{rs} \ra \PP^1$.
For this, note only that all theta characteristics  of a hyperelliptic curve are sums of ramification points of the
hyperelliptic covering (see for example \cite[Section III, 5]{m1}).

(ii): From the proof of Proposition \ref{p4.1} we know that the composition
$$
JY_r \times JY_{rs} \stackrel{\alpha}{\lra} P(Y/Y_s) \stackrel{\gamma}{\lra} JY_r \times JY_{rs}
$$
with $\gamma :=(\Nm \nu_r,\Nm \nu_{rs})$,
is multiplication by 2. If $\theta:= \theta_{JY \times JY_{rs}}$ denotes the canonical polarization of
$JY_r \times JY_{rs}$, this implies that $(\gamma \circ \alpha)^{-1}(\theta) = 4 \theta$ (see
\cite[Corollary 2.3.6]{bl}). Since $\alpha$ is an isomorphism, it follows that $\gamma^{-1}(\theta)$
is the fourth power  of a principal polarization, say $ \gamma^{-1}(\theta) = 4 \Xi$.

Now $\alpha: JY_r \times JY_{rs} \ra P(Y/Y_s)$ is an isomorphism. The canonical principal
polarization of $JY$ restricts to $\nu_r^*(JY_r)$ as twice a principal one, and to
$\nu_{rs}^*(JY_{rs})$ as twice a principal one, the restriction to $P(Y/Y_s)$ is $2\Xi$. Then (ii) follows from the fact that the map $\alpha$ is $G$-equivariant, since both
varieties are the eigen-subvarieties of $-1$ for the same element of $G$, namely $\sigma$.
\end{proof}

\end{document}